\documentclass[10pt,leqno]{amsart}
\usepackage{graphicx}
\usepackage{indentfirst,csquotes}
\usepackage{hyperref}
\usepackage{amsthm}

\usepackage{amssymb, amsmath}

\usepackage{xcolor,paralist,hyperref,etoolbox} % Удален `titlesec` и `fancyhdr`

\newtheorem{theorem}{Theorem}[]

\newtheorem{Lemma}{Lemma}

\newtheorem*{Stat}{Statement}

\theoremstyle{definition}
\newtheorem*{remark}{Remark}
\newtheorem{definition}{Definition}
\usepackage[style=numeric, backend=biber]{biblatex}
\hypersetup{ colorlinks=true, linkcolor=black, filecolor=black, urlcolor=black }

\usepackage{lipsum}
\begin{document}

\title[Computing the Topological Degree of Maps Between 2-Spheres]{Computing the Topological Degree of Maps Between 2-Spheres}

 %%%%%%%%%%%%
\author[D. Kucher]{Daniil Kucher}
\date{\today}
\address{Address}
\email{dakucher@edu.hse.ru}
\maketitle

\let\thefootnote\relax

\begin{abstract}

We describe an effective method for computing the topological degree of continuous functions $R:S^2 \to S^2$, where $S^2$ is the Riemann sphere. Our approach generalizes the degree formula for rational functions of complex polynomials, $\frac{f}{g}$, without common zeros. To apply our method, it is necessary to represent the function $R$ as the ratio of two continuous complex-valued functions $f$ and $g$ without common zeros. By using the Hopf fibration, this method reduces the problem to computing the winding number of a loop. This enables us to compute the degree of $\frac{f}{g}$ even when $f$ and $g$ are arbitrary continuous complex functions without common zeros, and the fraction has a limit at infinity (which can be finite or infinite). Specifically, if $f$ and $g$ are complex polynomials in $z$ and $\bar{z}$, and the highest-degree homogeneous component of the polynomial with the greater algebraic degree has a finite or infinite limit as $|z|\to\infty$, then the problem reduces to counting the roots of a complex polynomial inside the unit circle, obtained from this component.

\end{abstract}

$\,$

$\,$

\section{Introduction}

Consider the standard unit $n$-sphere $S^n \subset \mathbb{R}^{n+1}$ for $n \in \mathbb{Z}_{\ge 1}$ and the group $\pi_{n}(S^{n})$ of homotopy classes of continuous maps $f : S^n \to S^n$. This group is isomorphic to the group of integers, $\mathbb{Z}$, and the integer corresponding to the homotopy class of a map $f$ under this isomorphism is called the \textit{degree} of $f$. The identity map on $S^n$ is defined to have degree $1$.

In general, there is a standard method for computing the degree of a map. 
Let $R : S^{n} \to S^{n}$ be a continuous function, and let $p \in S^{n}$ be a \textit{regular} value. 
That is, there exists an open neighborhood $U$ of $p$ such that the preimage 
$R^{-1}(U)$ is a disjoint union of open sets $U_{\alpha}$, and $R$ restricted 
to each $U_{\alpha}$ is a homeomorphism $U_{\alpha} \to U$. 
For each $U_{\alpha}$, the restriction $R|_{U_{\alpha}}$ either preserves or 
reverses orientation. The degree of $R$ is then equal to the number of indices 
$\alpha$ for which $R|_{U_{\alpha}}$ preserves orientation, minus the number 
of indices $\beta$ for which $R|_{U_{\beta}}$ reverses orientation, noting that the number of preimages of $p$ is finite. 

However, the standard method may be difficult to apply even in the case $n=2$. For a map between Riemann spheres, defined by a rational function of $z$ and $\bar{z}$, finding the preimages of a point may require solving a system of high-degree polynomial equations. Furthermore, computing the sign of the Jacobian determinant at each solution can be computationally prohibitive. For example, let $R : S^{2} \to S^{2}$, where $S^{2}$ is identified with the Riemann sphere, 
be defined by the rational function
\[
R(z) = \frac{z\bar{z}^{4} + z\bar{z}^{2} + 3}{z^{3}\bar{z} + z},
\]  
with the convention that $R(p) = \infty$ if the denominator vanishes at $p$ or if $p = \infty$. 
Note that $R(z) \to \infty$ as $z \to \infty$, and that the numerator and denominator do not share any zeros; hence, the function is well-defined on $S^2$.
In Section~\ref{sec:computation}, we demonstrate how to compute the degree of this map and of other rational functions in $z$ and $\bar{z}$ effectively.

For the case $n=1$, several methods exist for computing the degree of a map. Consider the fundamental group $\pi_1(\mathbb{C} \setminus \{0\})$, which is isomorphic to $\mathbb{Z}$. We choose the isomorphism such that counter-clockwise loops correspond to positive integers. Then, for any continuous loop $\alpha: S^1 \to \mathbb{C} \setminus \{0\}$, we can define its \textit{winding number} (or \textit{index}) with respect to the origin as the 
image of this loop under this isomorphism. If the loop is smooth, the winding number is given by the integral formula:

\begin{equation}
\operatorname{ind}_0(\alpha) = \frac{1}{2\pi i} \oint_\alpha \frac{dz}{z} = \frac{1}{2\pi i} \int_0^{2\pi} \frac{\alpha'(t)}{\alpha(t)}dt \quad \label{eq:formula1}
\end{equation}

By embedding $S^1$ as the unit circle in $\mathbb{C} \setminus \{0\}$, we observe that the degree of any continuous map $S^{1} \rightarrow S^{1}$ coincides with its winding number.

In certain cases, formulas for the degree of 1-dimensional maps are particularly simple \cite{KhovanskiiBurda2008}. In this paper, we will demonstrate how to compute the degrees of maps from $S^2$ to $S^2$ by reducing the problem to computing the winding numbers of loops in $\mathbb{C} \setminus \{0\}$.

For the following theorem, we identify the sphere $S^2$ with the Riemann sphere $\bar{\mathbb{C}} = \mathbb{C} \cup \{\infty\}$ via stereographic projection.

\begin{theorem}
\label{Theorem1}

Let $f, g: \mathbb{C} \to \mathbb{C}$ be continuous functions with no common zeros. Assume that $\lim\limits_{|z|\to\infty} \frac{f(z)}{g(z)} = \infty$ and that $f(z) \ne 0$ for all $|z| \geq M$ for some constant $M > 0$. Consider the continuous map $R: \bar{\mathbb{C}} \to \bar{\mathbb{C}}$ defined as follows:
$$
R(z) = 
\begin{cases}
\frac{f(z)}{g(z)} & \text{if } g(z) \ne 0 \\
\infty & \text{if } g(z) = 0 \\
\infty & \text{if } z = \infty
\end{cases}
$$
Then, the degree of the map $R$ is equal to the winding number of the loop $\tilde{f}: S^1 \to \mathbb{C} \setminus \{0\}$, given by:
$$ \tilde{f}(\phi) = f(Me^{i\phi}) $$
\end{theorem}

The winding number of a loop can be calculated effectively using formula \eqref{eq:formula1} and numerical integration.

In particular, when $f(z,\bar{z})$ and $g(z,\bar{z})$ are complex polynomials in $z$ and $\bar{z}$, the problem reduces to counting the roots of a complex polynomial in the variable $z$
inside the open unit disk, under the assumption that the highest-degree homogeneous component of the polynomial with the larger degree (either \(f\) or \(g\)) tends to infinity as \(|z|\to\infty\). For this purpose, geometric methods for counting the roots of complex polynomials in closed regions may also be applied \cite{GarciaZapata2014}.

\begin{theorem}
\label{Theorem2}
Let $R(z) = \frac{f(z,\bar{z})}{g(z, \bar{z})}$ be a rational function where $f(z,\bar{z})$ and $g(z,\bar{z})$ are polynomials in $z$ and $\bar{z}$ with no common zeros. Assume, without loss of generality, that $\deg(f) \ge \deg(g)$. Let $d = \deg(f)$, and let $T(z, \bar{z})$ be the homogeneous polynomial of degree $d$ from the highest-degree terms of $f(z,\bar{z})$. Assume that $T$ has a limit (finite or infinite) at infinity.

Then the degree of the map $R: \bar{\mathbb{C}} \to \bar{\mathbb{C}}$ is equal to the number of roots of the polynomial $\tilde{T}(z) = z^d T(z, z^{-1})$ with modulus less than 1, counted with multiplicity, minus $d$.

Equivalently, the degree of the map is equal to the index of the curve
$$[0, 2\pi] \to \mathbb{C} \setminus \{0\}: \alpha \mapsto T(e^{i\alpha}, e^{-i\alpha})$$
\end{theorem}

If the function $\frac{f}{g}$ has a finite limit $a$ at infinity, we can then compose it with the Möbius transformation $z \mapsto \frac{1}{z-a}$:
$$
\frac{1}{\frac{f}{g}-a} = \frac{g}{f-ag}.
$$
This transformation is an automorphism of the Riemann sphere and has a degree of 1, so the composition does not change the degree of the function $\frac{f}{g}$.

\section*{Acknowledgements}
The author is grateful to his advisor, Vladlen Timorin, for valuable guidance and for suggesting the problem that initiated this research, and to Sergey Melikhov and Alexey Gorinov for helpful discussions and comments that contributed to clarifying several ideas presented in the paper.

\section{Proofs}

We identify the sphere \(S^2\) with the Riemann sphere $\bar{\mathbb{C}} = \mathbb{C} \cup \{\infty\}$, via stereographic projection. The sphere can also be viewed as the disk $D^2$ with its boundary collapsed to a single point. For convenience, in the case where a continuous function $R : \overline{\mathbb{C}} \to \overline{\mathbb{C}}$ is given as the ratio of two continuous complex-valued functions $f$ and $g$, we will write
\[
R(z) = \frac{f(z)}{g(z)},
\]
with the understanding that this convention also covers the cases when the denominator vanishes or when $z = \infty$.

\begin{definition}
Consider the set $\tilde{S} = \{\tilde{x} \in \mathbb{C} \mid |\tilde{x}| = 1\}$, which is a disjoint copy of the unit circle, where $\tilde{x}$ denotes the copy of a point $x \in S^1$. Let $D^{2} = \{x \in \mathbb{C} \mid |x| \leq 1\}$ be the closed unit disk.  
We define a topology on the set $\mathbb{C} \cup \tilde{S}$ so that the following map $f: D^2 \to \mathbb{C} \cup \tilde{S}$ is a homeomorphism:

$$
f(x) =
\begin{cases}
e^{i\arg(x)} \tan(\frac{\pi}{2}|x|), & |x| < 1 \\
\tilde{x}, & |x| = 1
\end{cases}
$$
where $\arg(x)$ is the argument of the point $x$, measured counterclockwise. We refer to the space $\mathbb{C} \cup \tilde{S}$ as \textit{plane with a boundary}.
\end{definition}

First, we prove that our upcoming result applies to any continuous function $R: \bar{\mathbb{C}} \to \bar{\mathbb{C}}$:
\begin{Lemma}
    Let $R: \bar{\mathbb{C}} \to \bar{\mathbb{C}}$ be a continuous function. Then there exist continuous functions $f,g: \mathbb{C} \to \mathbb{C}$ that do not have any common zeros, such that

    $$
        R = \frac{f}{g}
    $$
    on the domain where $g$ is non-zero.
\end{Lemma}

\begin{proof} Consider the Hopf fibration $h: S^3 \to S^2$. From the local triviality of this fibration it follows \cite[p.~107]{Fomenko} that any continuous map $R: S^2 \to S^2$ admits a \textit{lifting}, that is, a continuous map $\tilde{R}: D^2 \to S^3$ such that 
$$
h \circ \tilde{R} = R \circ \pi,
$$ 
where $D^2$ is the plane with a boundary and $\pi: D^2 \to S^2$ is the canonical quotient map obtained by collapsing the boundary of $D^2$ to a point.

The sphere $S^3$ is defined as the set of pairs $\{(z_1, z_2) \in \mathbb{C}^2 \mid |z_1|^2+|z_2|^2=1\}$. We denote the lifting as $\tilde{R}(p) = (R_1(p), R_2(p))$. According to the definition of the fibration, the original map $R$ is the ratio $R = \frac{R_1}{R_2}$.
 The condition 
 $$|R_1(p)|^2+|R_2(p)|^2=1$$
 ensures that $R_1$ and $R_2$ do not have any common zeros. Note that $R_1$ and $R_2$ are continuous functions from $D^2$ to $\mathbb{C}$. By restricting these maps to the interior of $D^{2}$, we define
\[
f := R_{1}|_{\operatorname{int}(D^{2})}, \qquad 
g := R_{2}|_{\operatorname{int}(D^{2})}.
\]

\end{proof}

We proceed to the proof of Theorem 1. For this, we will need the following definition:

\begin{definition}
Let $A$ be a subspace of a topological space $X$, and let $x_0$ be a base point in $A$. The relative homotopy group $\pi_2(X,A,x_0)$ is the set of homotopy classes of maps $f$ from the square $I^2 = [0,1] \times [0,1]$ to $X$ such that $f(\partial I^2) \subset A$ and $f(\{0,1\} \times [0,1] \cup [0,1] \times \{0\}) = \{x_0\}$, where $\partial I^2$ is the boundary of $I^2$. An element of this group is called \textit{a cubic relative spheroid}. The group operation is defined as follows:
$$ (f+g)(t_1,t_2) = \begin{cases} f(2t_1,t_2), & 0 \le t_1 \le 1/2 \\ g(2t_1-1,t_2), & 1/2 \le t_1 \le 1 \end{cases} $$
\end{definition}

Alternatively, we can define this group using maps from the disk $D^2 = \{z \in \mathbb{C} \mid |z| \le 1\}$ to $X$ such that $f(\partial D^2) \subset A$ and $f(S_l) = \{x_0\}$, where $S_l = \{z \in \partial D^2 \mid \operatorname{Re}(z) \le 0\}$ is the left semicircle. An element of the group defined via the disk is called \textit{a relative ball spheroid}.

\begin{Stat}
The operation $+$ defined above satisfies all group axioms. \textup{\cite[p.~97]{Fomenko}} 

\end{Stat}

In the following, let $D^{2}$ be the closed unit disk. We consider the sphere $S^2$ to be the quotient space $D^2/\partial D^2$, which is formed by collapsing the boundary $\partial D^2$ to a single point. This point serves as the basepoint of the sphere, denoted $x_0 = [\partial D^2]$. Any continuous map $f:(D^2,\partial D^2)\to(S^2,x_0)$ that sends the boundary $\partial D^2$ to the basepoint $x_0$ can be viewed as a map from $S^2$ to $S^2$ and represents an element of $\pi_2(S^2,x_0)$.

\begin{Lemma}
Let $p : S^3 \to S^2$ be a locally trivial fibration with fiber $T = p^{-1}(x_0)$ over the basepoint $x_0 \in S^2$. Let $y_0 \in T$ be a basepoint for the fiber. Consider a map $f : (D^2, \partial D^2) \to (S^2, x_0)$ representing a class in $\pi_2(S^2, x_0)$. Let $\tilde{f} : D^2 \to S^3$ be a lifting of $f$ (i.e., $p \circ \tilde{f} = f$) such that $\tilde{f}(d_0) = y_0$, where $d_0$ is a chosen basepoint in $\partial D^2$ and also in $D^2$. Then the mapping
$$
[f] \mapsto [\tilde{f}|_{\partial D^2}]
$$
is a well-defined isomorphism from $\pi_2(S^2, x_0)$ to $\pi_1(T, y_0)$.
\end{Lemma}

\begin{proof}Define a map $p_*: \pi_2(S^3, T, y_0) \to \pi_2(S^2, x_0)$ by $p_*([g]) = [p \circ g]$. It is a known result that $p_*$ is an isomorphism (\cite{Fomenko} pp. 116-117).

Consider the homotopy sequence for the pair $(S^3, T)$  with basepoint $y_0$:
$$ \dots \to \pi_2(S^3, y_0) \xrightarrow{i_*} \pi_2(S^3, T, y_0) \xrightarrow{k} \pi_1(T, y_0) \xrightarrow{j_*} \pi_1(S^3, y_0) \to \dots $$
Here, $i_*$ is induced by the inclusion $(S^3, y_0) \hookrightarrow (S^3, T)$, $k$ is the connecting homomorphism that restricts a relative spheroid to its boundary, and $j_*$ is induced by the inclusion $T \hookrightarrow S^3$. As this sequence is exact \cite[pp.~100--101]{Fomenko}, $k$ is an isomorphism.

We must now prove that for any lift $\tilde{f}$ satisfying the conditions of the lemma, the loop $\tilde{f}|_{\partial D^2}$ is homotopic to the loop $p_*^{-1}(f)|_{\partial D^2}$. This will show that the homotopy class of the boundary loop is independent of the choice of lift (provided it satisfies the conditions of the lemma) and that the mapping $[f] \mapsto [\tilde{f}|_{\partial D^2}]$  is an isomorphism, being a composition of isomorphisms $k \circ p_*^{-1}$.

We construct a homotopic relative ball spheroid for the map $\tilde{f}$.
Let $S_l \subset \partial D^2 \subset \mathbb{C}$ be the left semicircle, and let $d_0 = -1$ be our basepoint. There exists a homotopy $G_t: D^2 \to D^2$ for $t \in [0,1]$ such that:
\begin{itemize}
    \item $G_0 = \text{Id}$.
    \item $G_t(d_0) = d_0$ for all $t \in [0,1]$.
    \item $G_t(\partial D^2) \subset \partial D^2$ for all $t \in [0,1]$.
    \item $G_1(S_l) = \{d_0\}$.
\end{itemize}
Let $F_t = \tilde{f} \circ G_t$. The map $f' = \tilde{f} \circ G_1$ is now a relative ball spheroid, as it maps the arc $S_l$ to the point $\tilde{f}(d_0) = y_0$.
The homotopy $F_t$ preserves the basepoint and maps the boundary into the fiber, $F_t(\partial D^2) \subset T$. Thus, the boundary restrictions $\tilde{f}|_{\partial D^2}$ and $(\tilde{f} \circ G_1)|_{\partial D^2}$ are homotopic loops in $T$.

We have the following sequence of relations:
$$ [\tilde{f}] = [\tilde{f} \circ G_1] $$
Projecting down to $S^2$:
$$[f]= [p \circ \tilde{f}] = [p \circ (\tilde{f} \circ G_1)] \implies [f]= [p \circ f'] $$
And applying $p_{*}^{-1}:$
$$
p_{*}^{-1}([f])=[f' ]=[\tilde{f}]
$$
It follows that
$$
k\left(p_{*}^{-1}([f])\right) = [\tilde{f}|_{\partial D^{2}}]
$$
\end{proof}
Now, we can prove the first Theorem:

\begin{proof}[Proof of Theorem~\ref{Theorem1}]

We identify the sphere $S^2$ with the Riemann sphere $\overline{\mathbb{C}} = \mathbb{C} \cup \{\infty\}$, 
taking the point $\infty$ as the basepoint in both the domain and codomain of the map $R$.

From the assumption $\lim\limits_{|z|\rightarrow \infty}\frac{f(z)}{g(z)}=\infty$, it follows that such a number $M$ as required by the theorem exists. Consider a homotopy:
$$
G(z,t) = \begin{cases}
\frac{f\left(t\frac{Mz}{|z|}+(1-t)z\right)}{|f\left(t\frac{Mz}{|z|}+(1-t)z\right)|} |f(z)|, & |z| \geq M \\
f(z), & |z| < M
\end{cases}
$$
It is clear that $|G(z,t)| = |f(z)|$ for all $t \in [0,1]$. Moreover, if $|z|\ge M$, then for all $t\in[0,1]$ we have
$$
\left|t\frac{Mz}{|z|}+(1-t)z \right| \geq M.
$$
Therefore, the map $G$ is continuous. Furthermore, when $|z| \geq M$:

$$
\arg(G(z, 1)) = \arg f\left(\frac{z}{|z|}M\right)
$$
Since $|G(z,t)| = |f(z)|$ and this homotopy has no common zeros with the function $g$ for all $t \in [0,1]$, the next homotopy is also continuous on $S^{2} \times [0,1]$:
$$ \tilde{G}(z,t) = \frac{G(z,t)}{g(z)} $$

Denote
$$
f_1(z) = G(z,1)
$$

Therefore, the degree of the map $R_1 = \frac{f_{1}}{g}$ is equal to the degree of the map $R = \frac{f}{g}$. Without loss of generality, suppose that $\arg(f(M)) = 0$. Consider the Hopf fibration $h: S^3 \to S^{2}$:
$$
h(z_1, z_2) = \begin{cases}
\frac{z_1}{z_2}, & z_2 \neq 0 \\
\infty, & z_2 = 0
\end{cases}
$$
where $|z_1|^2 + |z_2|^2 = 1$.

Let $D^{2}$ denote the plane with a boundary. Take the basepoints $(1,0) \in S^3$ and $\tilde{1} \in \tilde{S}$, 
with $\tilde{S}$ denoting the boundary of the plane. Consider a map $\tilde{R}: \mathbb{C} \to S^3$:

$$
\tilde{R}(z) = \left( \frac{f_1(z)}{\sqrt{|f_1(z)|^2 + |g(z)|^2}}, \frac{g(z)}{\sqrt{|f_1(z)|^2 + |g(z)|^2}} \right)
$$
Let the fiber of the Hopf fibration over the point $\infty$ be denoted by 
$S^1 = \{(z,0) \mid |z|=1\}$. The map $\tilde{R}$ can be extended continuously to the boundary of the plane, $\partial D^{2}$: 

$$
\tilde{R}|_{\partial D^{2}} : \tilde{S} \to S^1, \quad
\tilde{R}|_{\partial D^{2}}(\tilde{z}) =\left( \frac{f_{1}(Mz)}{|f_{1}(Mz)|},0\right) =\left( \frac{f(Mz)}{|f(Mz)|}, 0 \right)
$$
where $\tilde{z} \in \tilde{S}$ corresponds to the point $z$ on the unit circle $|z|=1$. Note that this extended map will be a lifting for the map $R_{1}$, and we use the same notation $\tilde{R}$ for its extension.

Lemma 2 implies that the degree of the map $R_1$ is equal to the degree of the map $\tilde{R}|_{\partial D^{2}}$. 
Moreover, this degree coincides with the index of the loop $\tilde{f} : [0,2\pi] \to \mathbb{C} \setminus \{0\}$ defined by
$$
\tilde{f}(\phi) = f \big( Me^{i\phi} \big).
$$

\end{proof}
\begin{remark}
Below we provide a brief outline of an alternative proof of Theorem 1, demonstrating how the general case can be reduced to a special case with additional assumptions.

We first assume that the functions $f$ and $g$ are smooth, with their graphs transversal to the zero section, and that the zero set of $g$ is bounded. This ensures that the roots of $f$ and $g$ are isolated.

The proof relies on the theorem that the degree of a map is the sum of its \textit{local degrees} at the preimages of a regular value. The local degree at a point $p$ can be computed as the index of the image of a small, positively oriented loop around $p$, provided that the loop does not contain any other preimages.

In this special case, the degree of the map $R$ is given by the sum of local degrees:
$$
\deg(R) = \sum_{p \in R^{-1}(q)} \deg_p(R) = \deg_\infty(R) + \sum_{p \ne \infty}\deg_p(R)
$$
Here, the local degree at infinity is $\deg_\infty(R) = \operatorname{ind}_{\infty}(f) - \operatorname{ind}_{\infty}(g)$, while for any other preimage $p$, the local degree is $\deg_p(R) = \operatorname{ind}_p(g)$. Thus, by the argument principle, the sum of local degrees of $g$ is $\operatorname{ind}_{\infty}(g)$, leading to:
$$
\deg(R) = (\operatorname{ind}_{\infty}(f)-\operatorname{ind}_{\infty}(g)) + \operatorname{ind}_{\infty}(g) = \operatorname{ind}_{\infty}(f).
$$
If the zero set of $g$ is unbounded, one can reduce to the bounded case by choosing a sufficiently large disk whose complement is mapped by $R$ into a small neighborhood of infinity, with $g$ non-vanishing on the boundary, and then collapsing the complement and the neighborhood of infinity to points. This yields a new map of the same degree, with a finite set of zeros of $g$.

The general case can be reduced to this special case, provided that any continuous function can be approximated by a smooth function whose graph is transversal to the zero section.

\end{remark}

\begin{proof}[Proof of Theorem~\ref{Theorem2}]

The function $T$ has no zeros on the unit circle. Suppose, for contradiction, that $T(e^{i\alpha_0}) = 0$ for some $\alpha_0 \in [0, 2\pi]$. Since $T$ is homogeneous, we must have $T(z) = 0$ for all $z$ on the ray through $e^{i\alpha_0}$. This contradicts the existence of an infinite limit of $T$ at infinity.
Therefore, the minimum of the modulus of $T$ on the unit circle is greater than $0$. It then follows from the homogeneity of $T$ that there exists a sufficiently large number $M > 0$ such that for all $|z| > M$, the function $f(z)$ has no zeros, and
$$
\left|T(z)\right| > \left|f(z)-T(z)\right|.
$$
 It follows from Theorem \ref{Theorem1} that the degree of the function $R$ is equal to 
$\mathrm{ind}_0(\tilde{f})$, where

$$
\tilde{f} : [0,2\pi] \to \mathbb{C}\setminus \{0\}, \qquad \tilde{f}(\alpha) = f(Me^{i\alpha}).
$$

But since $|T(z)| > |f(z)-T(z)|$:
$$ \operatorname{ind}_{0}(\tilde{f}) = \operatorname{ind}_{0}(T(Me^{i\alpha})) $$

Due to the homogeneity of $T$ and the fact that the winding number is invariant under scaling by a non-zero constant, we have:
$$
\operatorname{ind}_{0}(T(Me^{i\alpha}))=\operatorname{ind}_{0}(M^{d}T(e^{i\alpha})) =\operatorname{ind}_{0}(T(e^{i\alpha}))
$$

By the definition of $\tilde{T}$,
$$ T(e^{i\alpha}) = e^{-id\alpha} \tilde{T}(e^{i\alpha}). $$
Applying the property that the index of a product of curves is the sum of their indices, we obtain:
$$ \operatorname{ind}_{0}(\tilde{f}) = 
 \operatorname{ind}_{0}(e^{-id\alpha}\tilde{T}(e^{i\alpha}))=\operatorname{ind}_{0}(e^{-id\alpha})+\operatorname{ind}_{0}(\tilde{T}(e^{i \alpha}))=-d+\operatorname{ind}_{0}(\tilde{T}(e^{i \alpha}))
 $$
Then, by the argument principle, the index of the curve $\tilde{f}$ is equal to the number of roots of the polynomial $\tilde{T}$ in the open unit disk, counted with multiplicity, minus $d$.

If $\deg(g) = \deg(f)$, let $a = \lim\limits_{|z|\to\infty}\frac{f(z)}{g(z)}$. The transformation $x \mapsto \frac{1}{x-a}$ has degree $1$, and
$$
\frac{1}{R-a} = \frac{g}{f-ag},
$$
where the right-hand side tends to infinity as $|z|\to\infty$. Hence $\deg R = \deg\frac{g}{f-ag}$, and Theorem~\ref{Theorem1} applies to $\frac{g}{f-ag}$.

\end{proof}

\section{Examples of Computation}
\label{sec:computation}
In all these examples, we want to compute the degree of the map $R = \frac{f}{g}$, where $f$ and $g$ are polynomials in $z$ and $\bar{z}$, with
$$
d = \max(\deg f, \deg g),
$$
and define $T(z,\bar{z})$ as the homogeneous polynomial of degree $d$ formed by the highest-degree terms of $f$ or $g$, whichever has the larger degree. We then define $\tilde{T}(z) = z^d T(z, z^{-1})$.
\\
\\
1) Let $f(z)=\sum_{t=0}^{k}a_{t}z^{t}$ and $g(z)=\sum_{t=0}^{m}b_{t}z^{t}$ with $m < k$, where $a_k \neq 0$ and $b_m \neq 0$. We assume that $f$ and $g$ have no common roots. In this case, the homogeneous polynomial of highest degree is $T(z)=a_k z^k$. The associated polynomial is $\tilde{T}(z)=a_k z^{2k}$, which has $2k$ roots with modulus less than 1 (all at $z=0$), counted with multiplicity. The degree of the map $f/g$ is therefore equal to 
\[2k-\deg(T)=2k-k=k.
\]

\medskip

2) Let $f(z)=\sum_{t=0}^{k}a_{t}\overline{z}^{t}$ and $g(z)=\sum_{t=0}^{m}b_{t}\overline{z}^{t}$ with $k < m$, where $a_k \neq 0$ and $b_m \neq 0$. We assume that $f$ and $g$ have no common roots.
In this case, the homogeneous polynomial of highest degree is $T(z)=b_m \overline{z}^{m}$. The associated polynomial is $\tilde{T}(z)=b_m$. Since $\tilde{T}$ is a nonzero constant, it has no roots with modulus less than 1. Hence, the degree of the map $R$ is 
\[
0 - \deg(T) = -m.
\]

\medskip

3) Let $f(z)=z\overline{z}^{4}+z\overline{z}^{2}+3$ and $g(z)=z^{3}\overline{z}+z$.
It is easy to check that $f(z)$ and $g(z)$ have no common zeros.
The homogeneous polynomial of highest degree is $T(z) = z\overline{z}^{4}$. The associated polynomial is $\tilde{T}(z) = z^{2}$. The polynomial $\tilde{T}$ has 2 roots with modulus less than 1, counted with multiplicity. Therefore, the degree of the map $R$ is
$$
\deg(R) = 2 - \deg(T) = 2 - 5 = -3.
$$

\medskip
4) Let $f(z)= z^{2}\bar{z}^{3}+2z^{4}\bar{z}+3z^{2}+2$ and $g(z)= 3z^{3}+\bar{z}$. We will show that $f$ and $g$ have no common roots. Assume for the sake of contradiction that there exists a common root $z \in \mathbb{C}$ such that
$$
\begin{cases}
z^{2}\bar{z}^{3}+2z^{4}\bar{z}+3z^{2}+2=0 \\
3z^{3}+\bar{z}=0
\end{cases}
$$
From the second equation, we have 

$$\bar{z} = -3z^3$$
Note that $z \neq 0$. Taking the modulus of both sides gives us $$|z| = 3|z|^3$$
which implies $|z| = \frac{1}{\sqrt{3}}$.
From the first equation, we have  $$|z^{2}\bar{z}^{3}+2z^{4}\bar{z}+3z^{2}| = 2$$
However, by the triangle inequality,
$$|z^{2}\bar{z}^{3}+2z^{4}\bar{z}+3z^{2}| \leq  3\cdot 3^{-5/2}+3 \cdot 3^{-1}< 2$$

The homogeneous polynomial of highest degree is $$T(z,\bar{z})=z^{2}\bar{z}^{3}+2z^{4}\bar{z}$$
This polynomial tends to infinity as $|z| \rightarrow \infty$.

The associated polynomial is 
$$\tilde{T}(z)=z^{4}+2z^{8}=z^{4}(1+2z^{4})$$ 
The polynomial $\tilde{T}$ has 8 roots inside the unit circle. Therefore, the degree of the map $R$ is $8-5=3$.

\medskip

5) Let $f(z) = z^{3} + \overline{z}^{3} + z$ and $g(z)=1$.
In this case, the homogeneous polynomial of highest degree is $T(z)=z^3+\overline{z}^3$. The limit $\lim_{|z|\to\infty} T(z)$ does not exist. Nevertheless, $\lim_{|z|\to\infty} R(z) = \infty$.
It is easy to check that for $|z| \geq 1$, $R(z) \neq 0$. The degree of $R$ is therefore equal to the winding number of the curve $\tilde{f}$:
$$\tilde{f}:[0,2 \pi] \rightarrow \mathbb{C}\setminus\{0\}, \qquad \tilde{f}(\phi) =f(e^{i\phi}).$$
By the formula for the winding number,
$$\operatorname{ind}_{0}(\tilde{f}) = \frac{1}{2\pi i}\oint_{\tilde{f}}\frac{dw}{w} = \frac{1}{2\pi i}\int_{0}^{2\pi}\frac{(3ie^{3it}-3ie^{-3it}+ie^{it})dt}{e^{3it}+e^{-3it}+e^{it}} = 1.$$

\printbibliography

\end{document}